\documentclass[12pt]{article}
\usepackage{amsmath,amsfonts,amssymb,amsthm,txfonts,hyperref,url,t1enc}
\newtheorem{theorem}{Theorem}
\newcommand{\N}{\mathbb N}
\newcommand{\Z}{\mathbb Z}
\newcommand{\R}{\mathbb R}
\newcommand{\C}{\mathbb C}
\newcommand{\K}{\textbf{\textit{K}}}

\begin{document}
\hfill
\vskip 0.9truecm
\noindent
\centerline{\Large A subset of ${\Z}^n$ whose non-computability leads}
\vskip 0.2truecm
\centerline{\Large to the existence of a Diophantine equation}
\vskip 0.2truecm
\centerline{\Large whose solvability is logically undecidable}
\vskip 0.9truecm
\centerline{\Large Apoloniusz Tyszka}
\vskip 0.9truecm
\noindent
{\bf Abstract.} For $\K \subseteq \C$,
let $B_n(\K)=\Bigl\{(x_1,\ldots,x_n) \in {\K}^n:$ for each $y_1,\ldots,y_n \in \K$ the conjunction
\[
\Bigl(\forall i \in \{1,\ldots,n\}~(x_i=1 \Longrightarrow y_i=1)\Bigr) ~\wedge
\]
\[
\Bigl(\forall i,j,k \in \{1,\ldots,n\}~(x_i+x_j=x_k \Longrightarrow y_i+y_j=y_k)\Bigr) ~\wedge
\]
\[
\forall i,j,k \in \{1,\ldots,n\}~(x_i \cdot x_j=x_k \Longrightarrow y_i \cdot y_j=y_k)
\]
implies that $x_1=y_1 \Bigr\}$. We claim that there is an algorithm that for every computable function
\mbox{$f:\N \to \N$} returns a positive integer $m(f)$, for which a second algorithm accepts on the input $f$
and any integer \mbox{$n \geq m(f)$}, and returns a tuple \mbox{$(x_1,\ldots,x_n) \in B_n(\Z)$} with \mbox{$x_1=f(n)$}.
We compute an integer tuple \mbox{$(x_1,\ldots,x_{20})$} for which the statement \mbox{$(x_1,\ldots,x_{20}) \in B_{20}(\Z)$}
is equivalent to an open Diophantine problem. We prove that if the set $B_n(\Z)$ \mbox{$(B_n(\N)$, $B_n(\N\setminus\{0\}))$}
is not computable for some $n$, then there exists a Diophantine equation whose solvability
in integers (non-negative integers, positive integers) is logically undecidable.
\newpage
For $\K \subseteq \C$,
let $B_n(\K)=\Bigl\{(x_1,\ldots,x_n) \in {\K}^n:$ for each $y_1,\ldots,y_n \in \K$ the conjunction
\[
\Bigl(\forall i \in \{1,\ldots,n\}~(x_i=1 \Longrightarrow y_i=1)\Bigr) ~\wedge
\]
\[
\Bigl(\forall i,j,k \in \{1,\ldots,n\}~(x_i+x_j=x_k \Longrightarrow y_i+y_j=y_k)\Bigr) ~\wedge
\]
\[
\forall i,j,k \in \{1,\ldots,n\}~(x_i \cdot x_j=x_k \Longrightarrow y_i \cdot y_j=y_k)
\]
implies that $x_1=y_1 \Bigr\}$.
\vskip 0.2truecm
\par
Each of the following two statements
\[
\Bigl(164,~165,~164 \cdot 165,~(164 \cdot 165)^2,
\]
\[
132,~133,~132 \cdot 133,~(132 \cdot 133)^2,
\]
\[
143,~144,~143 \cdot 144,~(143 \cdot 144)^2,~1 \Bigr) \in B_{13}(\N \setminus \{0\})
\]
and
\[
\Bigl(164,~165,~164 \cdot 165,~(164 \cdot 165)^2,
\]
\[
131,~132,~133,~132 \cdot 133,~(132 \cdot 133)^2,
\]
\[
142,~143,~144,~143 \cdot 144,~(143 \cdot 144)^2,~1\Bigr) \in B_{15}(\N)
\]
equivalently expresses that in the domain of positive integers only
the triples $(132,143,164)$ and $(143,132,164)$ solve the equation
\[
x^2(x+1)^2+y^2(y+1)^2=z^2(z+1)^2
\]
The last claim is still not proved, see \cite[p.~53]{Sierpinski1}.
\vskip 0.2truecm
\par
The following Lemma is a special case of the result presented in \cite[p.~3]{Skolem}.
\vskip 0.2truecm
\noindent
{\bf Lemma}~(\cite[p.~177,~Lemma 2.1]{Tyszka2}). {\em For each non-zero integer $x$ there exist integers
$a$, $b$ such that $ax=(2b-1)(3b-1)$.}
\begin{proof}
Write $x$ as $(2y-1) \cdot 2^m$, where $y \in \Z$ and $m \in \Z \cap [0,\infty)$.
Obviously, \hbox{$\frac{\textstyle 2^{2m+1}+1}{\textstyle 3} \in \Z$.}
By Chinese Remainder Theorem, we can find an integer $b$ such that
\hbox{$b \equiv y {\rm ~(mod~} 2y-1)$}
and \hbox{$b \equiv \frac{\textstyle 2^{2m+1}+1}{\textstyle 3} {\rm ~(mod~} 2^m)$.}
Thus, \hbox{$\frac{\textstyle 2b-1}{\textstyle 2y-1} \in \Z$} and
\hbox{$\frac{\textstyle 3b-1}{\textstyle 2^m} \in \Z$.}
Hence
\[
\frac{(2b-1)(3b-1)}{x}=\frac{2b-1}{2y-1} \cdot \frac{3b-1}{2^m} \in \Z
\]
\end{proof}
\par
Let $b=200526827$. Then, 
\[
667378345 \cdot (132 \cdot 133 \cdot 143 \cdot 144)=(2b-1)(3b-1)
\]
\begin{theorem}\label{the1}
The statement
\[
\Bigl(164 \cdot 165,~(164 \cdot 165)^2,~164,~165,
\]
\[
132,~133,~132 \cdot 133,~(132 \cdot 133)^2,
\]
\[
143,~144,~143 \cdot 144,~(143 \cdot 144)^2,~132 \cdot 133 \cdot 143 \cdot 144,
\]
\[
b,~2b,~2b-1,~3b-1,~(2b-1)(3b-1),\frac{(2b-1)(3b-1)}{132 \cdot 133 \cdot 143 \cdot 144},~1\Bigr) \in B_{20}(\Z)
\]
equivalently expresses that in the integer domain only the triples $(x,y,z) \in$
\[
\Bigl(\{-133,132\} \times \{-144,143\} \times \{-165,164\}\Bigr) \cup \Bigl(\{-144,143\} \times \{-133,132\} \times \{-165,164\}\Bigr)
\]
solve the system
\begin{displaymath}
\left\{
\begin{array}{rcl}
x^2(x+1)^2+y^2(y+1)^2 &=& z^2(z+1)^2 \\
x(x+1)y(y+1) &\neq& 0
\end{array}
\right.
\end{displaymath}
\end{theorem}
\newpage
\begin{proof}
The following {\sl MuPAD} code
\begin{quote}
\begin{verbatim}
y:=(132*133*143*144/64+1)/2:
z:=(2^13+1)/3:
print(`b=`):
b:=numlib::ichrem([y,z],[2*y-1,64]);
print(`(2b-1)(3b-1)/(132*133*143*144)=`):
(2*b-1)*(3*b-1)/(132*133*143*144);
A:=[164*165,(164*165)^2,164,165,
132,133,132*133,(132*133)^2,
143,144,143*144,(143*144)^2,132*133*143*144,
b,2*b,2*b-1,3*b-1,(2*b-1)*(3*b-1),
(2*b-1)*(3*b-1)/(132*133*143*144),1]:
print(`the triples [i,j,k] with i=<j and A[i]+A[j]=A[k]`):
for i from 1 to 20 do
for j from i to 20 do
for k from 1 to 20 do
if A[i]+A[j]=A[k] then print([i,j,k]) end_if:
end_for:
end_for:
end_for:
print(`the triples [i,j,k] with i=<j<20 and A[i]*A[j]=A[k]`):
for i from 1 to 19 do
for j from i to 19 do
for k from 1 to 20 do
if A[i]*A[j]=A[k] then print([i,j,k]) end_if:
end_for:
end_for:
end_for:
\end{verbatim}
\end{quote}
returns the output
\begin{verbatim}
                        `b=`
                      200526827
          `(2b-1)(3b-1)/(132*133*143*144)=`
                      667378345
 `the triples [i,j,k] with i=<j and A[i]+A[j]=A[k]`
                     [3, 20, 4]
                     [5, 20, 6]
                     [8, 12, 2]
                     [9, 20, 10]
                    [14, 14, 15]
                    [14, 16, 17]
                    [16, 20, 15]
`the triples [i,j,k] with i=<j<20 and A[i]*A[j]=A[k]`
                      [1, 1, 2]
                      [3, 4, 1]
                      [5, 6, 7]
                      [7, 7, 8]
                     [7, 11, 13]
                     [9, 10, 11]
                    [11, 11, 12]
                    [13, 19, 18]
                    [16, 17, 18]
\end{verbatim}
\noindent
At the start, the code computes the integer $b$ by applying the algorithm presented in the proof
of the Lemma. Next, the code computes \mbox{$\frac{{\textstyle (2b-1)(3b-1)}}{{\textstyle 132 \cdot 133 \cdot 143 \cdot 144}}$}.
The triples displayed on the output justify the equivalence.
\end{proof}
\newpage
\begin{theorem}\label{the2}
The statement
\[
\Bigl(328,~330,~328 \cdot 330,~(328 \cdot 330)^2,
\]
\[
264,~266,~264 \cdot 266,~(264 \cdot 266)^2,
\]
\[
286,~288,~286 \cdot 288,~(286 \cdot 288)^2,
\]
\[
250,~16,~4,~2,~1\Bigr) \in B_{17}(\N \setminus \{0\})
\]
equivalently expresses that in the domain of positive integers only the triples \mbox{$(250,~286,~328)$}
and \mbox{$(272,~264,~328)$} solve the equation
\[
(x+14)^2(x+16)^2+y^2(y+2)^2=z^2(z+2)^2
\]
The last claim about a possible solutions to the above equation can be equivalently formulated
thus: in the domain of integers greater than $1$, only the triples
\[
(10,~13,~14),~(13,~10,~14),~(265,~287,~329),~(287,~265,~329)
\]
solve the equation
\[
(x^2-1)^2+(y^2-1)^2=(z^2-1)^2
\]
For the last equation, no other solutions are known, see \cite[p.~68]{Sierpinski2}.
\end{theorem}
\begin{proof}
The following {\sl MuPAD} code
\begin{quote}
\begin{verbatim}
A:=[328,330,328*330,(328*330)^2,
264,266,264*266,(264*266)^2,
286,288,286*288,(286*288)^2,
250,16,4,2,1]:
print(`the triples [i,j,k] with i=<j and A[i]+A[j]=A[k]`):
for i from 1 to 17 do
for j from i to 17 do
for k from 1 to 17 do
if A[i]+A[j]=A[k] then print([i,j,k]) end_if:
end_for:
end_for:
end_for:
print(`the triples [i,j,k] with i=<j<17 and A[i]*A[j]=A[k]`):
for i from 1 to 16 do
for j from i to 16 do
for k from 1 to 17 do
if A[i]*A[j]=A[k] then print([i,j,k]) end_if:
end_for:
end_for:
end_for:
\end{verbatim}
\end{quote}
returns the output
\begin{verbatim}
 `the triples [i,j,k] with i=<j and A[i]+A[j]=A[k]`
                     [1, 16, 2]
                     [5, 16, 6]
                     [8, 12, 4]
                     [9, 16, 10]
                     [13, 14, 6]
                    [16, 16, 15]
                    [17, 17, 16]
`the triples [i,j,k] with i=<j<17 and A[i]*A[j]=A[k]`
                      [1, 2, 3]
                      [3, 3, 4]
                      [5, 6, 7]
                      [7, 7, 8]
                     [9, 10, 11]
                    [11, 11, 12]
                    [15, 15, 14]
                    [16, 16, 15]
\end{verbatim}
\noindent
The triples displayed on the output justify the first equivalence.
The second equivalence is obvious.
\end{proof}
\par
The sets \mbox{$B_n(\Z)$} contain very non-trivial integer tuples as it follows from the next theorem.
\begin{theorem}\label{the3}
There is an algorithm that for every computable function \mbox{$f:\N \to \N$}
returns a positive integer $m(f)$, for which a second algorithm accepts on the input $f$ and
any integer \mbox{$n \geq m(f)$}, and returns a tuple \mbox{$(x_1,\ldots,x_n) \in B_n(\Z)$} with \mbox{$x_1=f(n)$}.
\end{theorem}
\begin{proof}
The author proved in \cite{Tyszka1} that there is an algorithm that for every computable function
\mbox{$f:\N \to \N$} returns a positive integer $m(f)$, for which a second algorithm accepts
on the input $f$ and any integer \mbox{$n \geq m(f)$}, and returns a system
\[
S \subseteq \{x_i=1,~x_i+x_j=x_k,~x_i \cdot x_j=x_k: i,j,k \in \{1,\ldots,n\}\}
\]
such that $S$ is consistent over the integers and each integer tuple $(x_1,\ldots,x_n)$ that solves
$S$ satisfies $x_1=f(n)$. Let \mbox{$\leq_n$} denote the order on \mbox{${\Z}^n$} which ranks the
tuples \mbox{$(x_1,\ldots,x_n)$} first according to \mbox{${\rm max}(|x_1|,\ldots,|x_n|)$}
and then lexicographically. The ordered set \mbox{$({\Z}^n,\leq_n)$} is isomorphic to \mbox{$(\N,\leq)$}.
To find an integer tuple \mbox{$(x_1,\ldots,x_n)$}, we solve the system~$S$ by performing the brute-force
search in the order~$\leq_n$.
\end{proof}
\par
The presented results lead to the following Conjecture.
\vskip 0.2truecm
\noindent
{\bf Conjecture.} {\em For each sufficiently large $n$, the sets $B_n(\Z)$, $B_n(\N)$ and $B_n(\N\setminus\{0\})$
are not computable.}
\vskip 0.2truecm
\par
The conclusion of the following Theorem~\ref{the4} is unconditionally true and well-known as the corollary
of the negative solution to Hilbert's Tenth Problem, see \mbox{\cite[p.~231]{Klee}}.
\begin{theorem}\label{the4}
If the set $B_n(\Z)$ is not computable for some $n$, then there exists a Diophantine
equation whose solvability in integers is logically undecidable.
\end{theorem}
\begin{proof}
By the Lemma, for each integers $a_1$, $y_1$ the statement $a_1 \neq y_1$ is equivalent to
\[
\exists a,b \in \Z ~a(a_1-y_1)-(2b-1)(3b-1)=0
\]
To an integer tuple $(a_1,\ldots,a_n)$ we assign the equation
\[
D_{\textstyle (a_1,\ldots,a_n)}(a,b,y_1,\ldots,y_n)=(a(a_1-y_1)-(2b-1)(3b-1))^2+
\sum_{\stackrel{\textstyle i \in \{1,\ldots,n\}}{\textstyle a_i=1}} (y_i-1)^2+
\]
\[
\sum_{\stackrel{\textstyle (i,j,k) \in \{1,\ldots,n\}^3}{\textstyle a_i+a_j=a_k}} (y_i+y_j-y_k)^2+
\sum_{\stackrel{\textstyle (i,j,k) \in \{1,\ldots,n\}^3}{\textstyle a_i \cdot a_j=a_k}} (y_i \cdot y_j-y_k)^2=0
\]
For each integers \mbox{$a_1,\ldots,a_n$}, the tuple \mbox{$(a_1,\ldots,a_n)$} does not belong to
\mbox{$B_n(\Z)$} if and only if the equation \mbox{$D_{\textstyle (a_1,\ldots,a_n)}(a,b,y_1,\ldots,y_n)=0$}
has a solution in integers \mbox{$a,b,y_1,\ldots,y_n$}. We prove that there exists an integer tuple
\mbox{$(a_1,\ldots,a_n)$} for which the solvability of the equation
\mbox{$D_{\textstyle (a_1,\ldots,a_n)}(a,b,y_1,\ldots,y_n)=0$} in integers \mbox{$a,b,y_1,\ldots,y_n$}
is logically undecidable. Suppose, on the contrary, that for each integers \mbox{$a_1,\ldots,a_n$}
the solvability of the equation \mbox{$D_{\textstyle (a_1,\ldots,a_n)}(a,b,y_1,\ldots,y_n)=0$}
can be either proved or disproved. This would yield the following algorithm for deciding whether
an integer tuple \mbox{$(a_1,\ldots,a_n)$} belongs to $B_n(\Z)$: examine all proofs
(in order of length) until for the equation \mbox{$D_{\textstyle (a_1,\ldots,a_n)}(a,b,y_1,\ldots,y_n)=0$}
a proof that resolves the solvability question one way or the other is found.
\end{proof}
\par
Similarly, but simpler, if the set $B_n(\N)$ ($B_n(\N\setminus\{0\})$) is not computable for some~$n$,
then there exists a Diophantine equation whose solvability in non-negative integers (positive integers)
is logically undecidable.

\noindent
Apoloniusz Tyszka\\
Technical Faculty\\
Hugo Ko\l{}\l{}\k{a}taj University\\
Balicka 116B, 30-149 Krak\'ow, Poland\\
E-mail address: \url{rttyszka@cyf-kr.edu.pl}
\end{document}